\documentclass[a4paper]{article}

\pdfoutput=1

\usepackage{amsmath}
\usepackage{amssymb}
\usepackage{amsthm}
\usepackage{mathtools}

\usepackage{tikz}
\usepackage{tikz-cd}

\usepackage{hyperref}
\usepackage{cleveref}

\usepackage{microtype}

\newcommand\ord[1]{[#1]}

\newcommand\cat[1]{\mathcal{#1}}

\newcommand\email[1]{\href{mailto:#1}{\texttt{#1}}}

\DeclarePairedDelimiter\set{\{}{\}}
\DeclarePairedDelimiter\abs{\lvert}{\rvert}
\DeclarePairedDelimiter\ang{\langle}{\rangle}

\DeclareMathOperator\colim{colim}

\DeclareMathOperator\cts{cts}

\DeclareMathOperator\id{id}

\DeclareMathOperator\op{op}

\DeclareMathOperator\yo{\mathsf{y}}

\DeclareMathOperator\Set{\mathsf{Set}}
\DeclareMathOperator\Pos{\mathsf{Pos}}
\DeclareMathOperator\Tos{\mathsf{Tos}}

\newtheorem{theorem}{Theorem}

\newtheorem{corollary}[theorem]{Corollary}
\newtheorem{lemma}[theorem]{Lemma}
\newtheorem{proposition}[theorem]{Proposition}

\theoremstyle{definition}
\newtheorem{definition}[theorem]{Definition}
\newtheorem{example}[theorem]{Example}

\begin{document}

\title{Partial orders are the free conservative cocompletion of total orders}
\author{Calin Tataru \\ \small University of Cambridge \\ \small \email{calin.tataru@cl.cam.ac.uk}}
\maketitle

\begin{abstract}
We show that the category of partially ordered sets $\Pos$ is equivalent to the free conservative cocompletion of the category of finite non-empty totally ordered sets $\Delta$, which is also known as the simplex category.
\end{abstract}

\section{Introduction}

Colimits are an important tool in category theory, allowing us to glue together objects in a category in a universal way.
However, most categories do not have all colimits.
We can get around this, for a small category $\cat C$, by considering its free cocompletion $\widehat{\cat C} = [\cat C^{\op}, \Set]$, which is cocomplete and contains $\cat C$ as a full subcategory via the Yoneda embedding:
\begin{align*}
\yo : \cat C &\to \widehat{\cat C} \\
c &\mapsto \cat C(-, c)
\end{align*}
Moreover, it satisfies the following universal property~\cite{lambek1966completions}.
For every cocomplete category $\cat D$ and functor $F : \cat C \to \cat D$, there is an essentially unique cocontinuous functor extending $F$ along the Yoneda embedding:
\[
\begin{tikzcd}
\cat C \drar["F"'] \rar["\yo"] & \widehat{\cat C} \dar[dashed, "\widehat{F}"] \\
& \cat D
\end{tikzcd}
\]

However, the category $\cat C$ will often have some colimits to start with, but the Yoneda embedding will not, in general, preserve those colimits.
This motivates the idea of the free conservative cocompletion.
We recall the definition from~\cite{velebil2002remark}.

\begin{definition}
\label{def:cocompletion}
The \emph{free conservative cocompletion} of a category $\cat C$ consists of:
\begin{itemize}
\item a cocomplete category $\widetilde{\cat C}$, and
\item a fully faithful cocontinuous functor $I : \cat C \to \widetilde{\cat C}$.
\end{itemize}
such that for every cocomplete category $\cat D$ and cocontinuous functor $F : \cat C \to \cat D$, there exists an essentially unique cocontinuous functor $\widetilde{F} : \widetilde{\cat C} \to \cat D$ such that:
\[
\begin{tikzcd}
\cat C \drar["F"'] \rar["I"] & \widetilde{\cat C} \dar[dashed, "\widetilde{F}"] \\
& \cat D
\end{tikzcd}
\]
\end{definition}

There is a well-known way to characterise the free conservative cocompletion of a small category, which is guaranteed to exist, due to Kelly~\cite[Theorem 6.23]{kelly1982basic}.
See~\cite[Theorem 11.5]{fiore1996enrichment} for a simpler description of the result without proof.

\begin{proposition}
\label{prop:cocompletion}
If $\cat C$ is a small category, the free conservative cocompletion $\widetilde{\cat C}$ is equivalent to the full subcategory of $[\cat C^{\op}, \Set]$ whose objects are the continuous presheaves (i.e.\ presheaves that take colimits in $\cat C$ to limits in $\Set$).
\end{proposition}

While this description is useful, it is not always easy to work with.
In general, obtaining a concrete description of the free conservative cocompletion of a given category is not straightforward.
In this paper, we will prove the following result.

\begin{theorem}
The category of partially ordered sets $\Pos$ is the free conservative cocompletion of the category of finite non-empty totally ordered sets $\Delta$.
\end{theorem}

\noindent
The proof makes use of the nerve of the inclusion $\Delta \hookrightarrow \Pos$, which is a functor
\[
N : \Pos \to [\Delta^{\op}, \Set]
\]
We first show that the inclusion is cocontinuous, so the image of the nerve is contained in the category of continuous simplicial sets.
We then show that the inclusion is dense, so the nerve is fully faithful.
Finally, we show that the nerve is essentially surjective onto the category of continuous simplicial sets.
Therefore $\Pos$ is equivalent to the free conservative cocompletion of $\Delta$ by \Cref{prop:cocompletion}.

\subsection{Related work}

A similar result was proved by Mimram and Di Giusto~\cite{mimram2013categorical}.
They give a concrete description of the free finite conservative cocompletion of a category $\mathcal{L}$ that has the same objects as $\Delta$ but different morphisms (partial strictly monotone maps instead of monotone maps).
There are some similarities in the proofs, mostly in the proof of transitivity in \Cref{lma:transitivity}, but our result is more general.

The motivation for the main result comes from associative $n$-categories~\cite{dorn2018associative,reutter2019high}.
The terms in an associative $n$-category are defined inductively over $\Delta$, yet several results use colimits and require passing to $\Pos$, such as~\cite{sarti2023posetal,tataru2023layout,tataru2024theory}.
We previously lacked a formal justification for this passage, and this paper finally provides one.

\subsection{Acknowledgements}

The author would like to thank his supervisor Jamie Vicary for reviewing this paper, as well as Alex Rice and Ioannis Markakis for helpful discussions.

\section{Preliminaries}

We first recall some basic definitions and facts from order theory.

\begin{definition}
A \emph{partial order} on a set $X$ is a relation $\leq$ that is:
\begin{itemize}
\item \emph{reflexive}, i.e.\ $x \leq x$ for all $x \in X$,
\item \emph{transitive}, i.e.\ if $x \leq y$ and $y \leq z$ then $x \leq z$ for all $x, y, z \in X$, and
\item \emph{antisymmetric}, i.e.\ if $x \leq y$ and $y \leq x$ then $x = y$ for all $x, y \in X$.
\end{itemize}
A \emph{total order} is a partial order such that either $x \leq y$ or $y \leq x$ for all $x, y \in X$.
\end{definition}

We write $\Pos$ for the category of partially ordered sets and monotone maps, $\Tos$ for the full subcategory of totally ordered sets, and $\Delta$ for the full subcategory of finite non-empty totally ordered sets, also known as the \emph{simplex category}:
\[
\Delta \hookrightarrow \Tos \hookrightarrow \Pos
\]

For convenience, we will often work with a skeletal presentation of the simplex category $\Delta$, where the objects are given by the finite non-empty ordinals
\[
\ord n = \set{0, 1, \dots, n}
\]
and the morphisms are generated by two families of monotone maps:
\begin{itemize}
\item \emph{face maps} $\delta_i : \ord{n - 1} \to \ord n$ skipping an element $i \in \ord n$, i.e.\
\[
\delta_i(j) = \begin{cases}
j & \text{if } j < i \\
j + 1 & \text{if } j \geq i
\end{cases}
\]
\item \emph{degeneracy maps} $\sigma_i : \ord{n + 1} \to \ord n$ duplicating an element $i \in \ord n$, i.e.\
\[
\sigma_i(j) = \begin{cases}
j & \text{if } j \leq i \\
j - 1 & \text{if } j > i
\end{cases}
\]
\end{itemize}
subject to the following equations, which are known as the \emph{simplicial identities}:
\begin{align*}
\delta_j \delta_i &= \delta_i \delta_{j - 1} & (i < j) \\
\sigma_j \sigma_i &= \sigma_i \sigma_{j + 1} & (i \leq j) \\
\sigma_j \delta_i &= \delta_j \sigma_{i - 1} & (i < j) \\
\sigma_j \delta_i &= \id & (i = j \text{ or } i = j + 1) \\
\sigma_j \delta_i &= \delta_{j - 1} \sigma_i & (i > j + 1)
\end{align*}

The category $\Pos$ is cocomplete, with its colimits obtained as follows: take the colimit in $\Set$, endow it with the smallest preorder $\leq$ making all maps monotone, and then take the quotient under the smallest equivalence relation $\sim$ such that:\footnotemark{}
\[
x \leq y \text{ and } y \leq x \implies x \sim y
\]
On the other hand, the category $\Delta$ is \emph{not} cocomplete (e.g.\ it has no coproducts).
However, we will see later that the inclusion $\Delta \hookrightarrow \Pos$ is cocontinuous, and in fact, the colimits in $\Delta$ are computed in the same way as the colimits in $\Pos$.

\footnotetext{The equivalence classes of $\sim$ are the strongly connected components of the preorder $\leq$.}

\begin{definition}
A \emph{linear extension} of a partial order $\leq$ is a total order $\preceq$ on the same set such that $\leq$ is contained in $\preceq$, i.e.\ $x \leq y$ implies $x \preceq y$ for all $x, y$.
\end{definition}

\begin{proposition}[Order extension principle~\cite{szpilrajn1930extension}]
Every partial order has a linear extension, and moreover, it is the intersection of all of its linear extensions.
\end{proposition}

\noindent
Note that for infinite sets, this requires Zorn's Lemma (which is equivalent to the axiom of choice).
However, for finite sets, it can be proved without choice.

\subsection{Simplicial sets}

Recall that the free cocompletion of the simplex category $\Delta$ is the category of simplicial sets $[\Delta^{\op}, \Set]$, which will play a crucial role in our proof.

\begin{definition}
A \emph{simplicial set} is a presheaf on $\Delta$, i.e.\ a functor $X : \Delta^{\op} \to \Set$.
\end{definition}

Given a simplicial set $X$, we adopt the following notation:
\begin{itemize}
\item $X_n$ is the image of $\ord n$, whose elements are called \emph{$n$-simplices},
\item $d_i : X_n \to X_{n - 1}$ is image of the face map $\delta_i : \ord{n - 1} \to \ord n$, and
\item $s_i : X_n \to X_{n + 1}$ is image of the degeneracy map $\sigma_i : \ord{n + 1} \to \ord n$.
\end{itemize}
In fact, the data of a simplicial set $X$ is completely determined by the sets $X_n$ and maps $d_i, s_i$ satisfying the dual of the simplicial identities (see \Cref{fig:simplicial-set}).

\begin{figure}
\centering
\begin{tikzcd}
X_0 \rar["s_0" description] &
X_1 \lar[shift left=1em, "d_0" description] \lar[shift right=1em, "d_1" description] \rar[shift left=1em, "s_1" description] \rar[shift right=1em, "s_0" description] &
X_2 \lar[shift left=2em, "d_0" description] \lar["d_1" description] \lar[shift right=2em, "d_2" description] \quad\cdots
\end{tikzcd}
\caption{The data of a simplicial set $X$.}
\label{fig:simplicial-set}
\end{figure}

\subsection{Nerve, dense functors}

\begin{definition}
Any functor $F : \cat C \to \cat D$ induces a functor $N_F : \cat D \to [\cat C^{\op}, \Set]$, called the \emph{nerve} of $F$, given by the restricted Yoneda embedding:
\[
\cat D \xrightarrow{\yo} [\cat D^{\op}, \Set] \xrightarrow{(-) \circ F^{\op}} [\cat C^{\op}, \Set]
\]
In particular, it sends every object $d \in \cat D$ to the presheaf $N_F(d) \coloneqq \cat D(F(-), d)$.
\end{definition}

Recall the following proposition due to Ulmer~\cite[Lemma 1.7]{ulmer1968properties}.

\begin{proposition}
\label{prop:dense}
For every functor $F : \cat C \to \cat D$, the following are equivalent:
\begin{itemize}
\item Every object $d \in \cat D$ is a colimit of objects in the image of $F$:
\[
d \cong \colim (F \downarrow d \xrightarrow{\pi_\cat C} \cat C \xrightarrow{F} \cat D)
\]
Here $F \downarrow d$ is the comma category whose objects are morphisms $F(c) \xrightarrow{f} d$ and whose morphisms are morphisms $\alpha : c_1 \to c_2$ making this commute:
\[
\begin{tikzcd}[column sep=small]
F(c_1) \drar["f_1"'] \ar[rr, "F \alpha"] && F(c_2) \dlar["f_2"] \\
& d &
\end{tikzcd}
\]
Also the functor $\pi_\cat C$ is the canonical projection sending $F(c) \xrightarrow{f} d$ to $c$.
\item The nerve $N_F : \cat D \to [\cat C^{\op}, \Set]$ is fully faithful.
\end{itemize}
\end{proposition}

A functor satisfying the conditions of \Cref{prop:dense} is called \emph{dense}.

\begin{example}
The Yoneda embedding $\yo : \cat C \to [\cat C^{\op}, \Set]$ is dense.
To see that, note that any presheaf is a colimit of representables, and the nerve is the identity.
\end{example}

\section{Main results}

We begin by showing that the inclusion $i : \Delta \hookrightarrow \Pos$ is cocontinuous.
We break it down into two steps: (1) we show that the inclusion $\Delta \hookrightarrow \Tos$ is cocontinuous, and (2) we show that the inclusion $\Tos \hookrightarrow \Pos$ is cocontinuous.

\begin{proposition}
\label{prop:split-mono}
Let $f : N \to T$ be an injective monotone map between totally ordered sets with $N$ finite and non-empty.
Then $f$ is a split monomorphism.
\end{proposition}

\begin{proof}
Write $N = \set{x_0, \dots, x_n}$.
We define $g : T \to N$ as follows:
\[
g(t) = \begin{cases}
x_0 & \text{if } t < f(x_0) \\
x_i & \text{if } f(x_i) \leq t < f(x_{i + 1}) \\
x_n & \text{if } t \geq f(x_n)
\end{cases}
\]
It is easy to see that $g \circ f = \id$, so $f$ is a split monomorphism.
\end{proof}

The following is adapted from a proof by David Gao on Mathoverflow.\footnote{See \url{https://mathoverflow.net/q/467739/525267}.}

\begin{proposition}
The inclusion $\Delta \hookrightarrow \Tos$ is cocontinuous.
\end{proposition}

\begin{proof}
Let $D : \cat J \to \Delta$ be a diagram such that it has a colimit $\phi : D \Rightarrow N$ for a finite non-empty totally ordered set $N$.
We claim that $\phi$ is also the colimit of $D$ in $\Tos$, so let $\psi : D \Rightarrow T$ be a cocone in $\Tos$ for a totally ordered set $T$.

Now consider the image of $\psi$ which is defined to be the following subset of $T$:
\[
\psi[D] \coloneqq \bigcup_{j \in \cat J} \psi_j[D_j] \subseteq T
\]
Note that $\cat J$ is non-empty since $\Delta$ has no initial object, so $\psi[D]$ is non-empty.
We claim that it is also finite.
Now for every finite non-empty subset $S \subseteq \psi[D]$, there is a map $r : \psi[D] \to S$ by \Cref{prop:split-mono}.
Hence $r \circ \psi$ is a cocone in $\Delta$, so there exists a unique monotone map $u : N \to S$ making this commute:
\[
\begin{tikzcd}
& S & \\
N \urar["u"] && \psi[D] \ular["r"'] \\
& D_j \ular["\phi_j"] \urar["\psi_j"'] &
\end{tikzcd}
\]
Note that $\phi$ and $\psi$ are jointly epic and $r$ is epic, so $u$ is also epic.
Thus $\abs{S} \leq \abs{N}$.
Now this holds for every finite non-empty subset $S \subseteq \psi[D]$, so $\abs{\psi[D]} \leq \abs{N}$.

Finally, we have the following isomorphism of sets of monotone maps:
\[
\set{v: N \to T \mid v \circ \phi = \psi} \cong \set{v : N \to \psi[D] \mid v \circ \phi = \psi}
\]
This is because $\phi$ is jointly epic so the image of $v$ equals $\psi[D]$.
Since $\psi[D]$ lives in $\Delta$, the universal property in $\Tos$ follows from the universal property in $\Delta$.
\end{proof}

\begin{proposition}
The inclusion $\Tos \hookrightarrow \Pos$ is cocontinuous.
\end{proposition}

\begin{proof}
Let $D : \cat J \to \Tos$ be a diagram such that it has a colimit $\phi : D \Rightarrow T$ for a totally ordered set $T$.
We claim that $\phi$ is also the colimit of $D$ in $\Pos$.

Now $\Pos$ is cocomplete, so $D$ has a colimit $\psi : D \Rightarrow P$ in $\Pos$ for a poset $P$, and so there exists a unique monotone map $u : P \to T$ making this commute:
\[
\begin{tikzcd}
P \ar[rr, "u"] && T \\
& D_j \ular["\psi_j"] \urar["\phi_j"'] &
\end{tikzcd}
\]
Let $i : P \hookrightarrow L$ be a linear extension of $P$.
Then $i \circ \psi$ is a cocone over $D$ in $\Tos$, and so there exists a unique monotone map $v : T \to L$ making this commute:
\[
\begin{tikzcd}
& L & \\
T \urar["v"] && P \ular["i"'] \\
& D_j \ular["\phi_j"] \urar["\psi_j"'] &
\end{tikzcd}
\]
We have that $v \circ u \circ \psi = v \circ \phi = i \circ \psi$.
Since $\psi$ is jointly epic, it follows that $v \circ u = i$.
In other words, every linear extension $L$ of $P$ factors through $T$.

This implies that $u$ must be order-reflecting: if $u(x) \leq u(y)$ in $T$, then $x \leq y$ in every linear extension $L$ of $P$, so $x \leq y$ in $P$.
Therefore $P$ is totally ordered.
Now the inclusion $\Tos \hookrightarrow \Pos$ reflects colimits as it is fully faithful, so $T \cong P$.
\end{proof}

\begin{proposition}
\label{prop:delta-cocontinuous}
The inclusion $i : \Delta \hookrightarrow \Pos$ is cocontinuous.
\end{proposition}

\begin{proof}
This follows immediately from the previous two propositions.
\end{proof}

\begin{proposition}
\label{prop:delta-dense}
The inclusion $i : \Delta \hookrightarrow \Pos$ is dense.
\end{proposition}

\begin{proof}
Let $P$ be a poset.
The comma category $i \downarrow P$ consists of:
\begin{itemize}
\item objects: monotone maps of the form
\[
x : \ord n \to P \qquad (n \in \mathbb{N})
\]
which are equivalent to finite chains of $P$.
\item morphisms: commutative triangles of the form
\[
\begin{tikzcd}
\ord n \ar[rr, "f"] \ar[dr, "x"'] && \ord m \ar[dl, "y"] \\
& P &
\end{tikzcd}
\]
which are equivalent to inclusions of chains, i.e.\ $x_i = y_{f(i)}$.
\end{itemize}
It is easy to see that $P$ is the colimit of the diagram $i \downarrow P \to \Delta \hookrightarrow \Pos$ because the colimit is just a union and every poset is equal to the union of its chains.
\end{proof}

Therefore by \Cref{prop:dense}, the nerve functor is fully faithful:
\begin{align*}
N : \Pos &\hookrightarrow [\Delta^{\op}, \Set] \\
P &\mapsto \Pos(i(-), P)
\end{align*}
For every poset $P$, the nerve $NP$ is a simplicial set whose $n$-simplices are the chains of length $n$ in $P$, i.e.\ tuples $(x_0, \dots, x_n)$ such that $x_i \leq x_{i + 1}$ for all $i$.
The face and degeneracy maps are given by applying transitivity and reflexivity:
\begin{align*}
d_i(x_0, \dots, x_n) &= (x_0, \dots, x_{i - 1}, x_{i + 1}, \dots, x_n) \\
s_i(x_0, \dots, x_n) &= (x_0, \dots, x_i, x_i, \dots, x_n)
\end{align*}

\begin{proposition}
The nerve $NP$ is continuous for every poset $P$.
\end{proposition}

\begin{proof}
Note that $NP$ is given by the composite
\[
\Delta^{\op} \xrightarrow{i^{\op}} \Pos^{\op} \xrightarrow{\yo P} \Set
\]
We have that $i^{\op}$ is continuous because $i$ is cocontinuous by \Cref{prop:delta-cocontinuous}, and $\yo P$ is continuous because hom-functors preserve limits in the first argument.
\end{proof}

Let $[\Delta^{\op}, \Set]_{\cts}$ be the full subcategory of continuous simplicial sets.
Hence the nerve functor exhibits $\Pos$ as a full subcategory of $[\Delta^{\op}, \Set]_{\cts}$, so we have
\[
N : \Pos \hookrightarrow [\Delta^{\op}, \Set]_{\cts}
\]
We claim that this is essentially surjective and hence an equivalence.
In particular, we will show that every continuous simplicial set arises as the nerve of a poset.
From now on, suppose that $X$ is a continuous simplicial set.

\begin{lemma}
\label{lma:relation}
The map $\ang{d_1, d_0} : X_1 \to X_0 \times X_0$ is injective.
\end{lemma}

\begin{proof}
The following diagram is a colimit in $\Delta$:
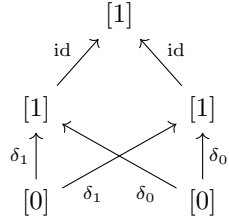
\[
\begin{tikzcd}[column sep=small]
& \ord 1 & \\
\ord 1 \urar["\id"] && \ord 1 \ular["\id"'] \\
\ord 0 \uar["\delta_1"] \ar[urr, "\delta_1"' very near start] && \ord 0 \ar[ull, "\delta_0" very near start] \uar["\delta_0"']
\end{tikzcd}
\]
Hence $X$ takes it to the following limit diagram in $\Set$:
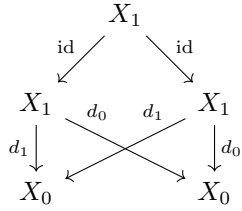
\[
\begin{tikzcd}[column sep=small]
& X_1 \dlar["\id"'] \drar["\id"] & \\
X_1 \dar["d_1"'] \ar[drr, "d_0" very near start] && X_1 \ar[dll, "d_1"' very near start] \dar["d_0"] \\
X_0 && X_0
\end{tikzcd}
\]
This is equivalent to the following diagram being a pullback:
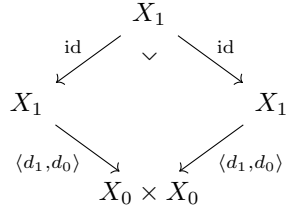
\[
\begin{tikzcd}[column sep=small]
& X_1 \dlar["\id"'] \drar["\id"] \ar[dd, phantom, "\llcorner" rotate=45, very near start] & \\
X_1 \drar["\ang{d_1, d_0}"'] && X_1 \dlar["\ang{d_1, d_0}"] \\
& X_0 \times X_0 &
\end{tikzcd}
\]
which is equivalent to the map $\ang{d_1, d_0} : X_1 \to X_0 \times X_0$ being injective.
\end{proof}

Hence we can view $X_1$ as a relation on $X_1$.
We write $x \leq_X y$ iff $(x, y) \in X_1$.
We claim that $\leq_X$ is in fact a partial order, and that $X$ is the nerve of $(X_0, \leq_X)$.

\begin{lemma}
\label{lma:chains}
The set $X_n$ is isomorphic to the following set:
\[
\set{(x_0, \dots, x_n) \in X_0^n \mid x_i \leq_X x_{i + 1}}
\]
\end{lemma}

\begin{proof}
This is true by definition for $n = 0, 1$.
Now note that $\ord{n + 2}$ arises as the following colimit in $\Delta$, with $n + 2$ occurrences of $\ord 1$ and $n + 1$ occurrences of $\ord 0$:
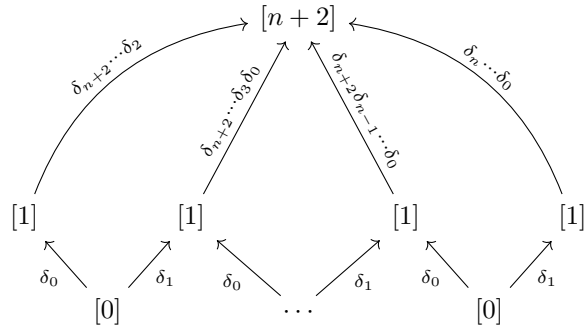
\[
\begin{tikzcd}[column sep=small]
&&& \ord{n + 2} &&& \\[4em]
\ord 1 \ar[urrr, bend left, "\delta_{n + 2} \cdots \delta_2" sloped] && \ord 1 \urar["\delta_{n + 2} \cdots \delta_3 \delta_0" sloped] && \ord 1 \ular["\delta_{n + 2} \delta_{n - 1} \cdots \delta_0" sloped] && \ord 1 \ar[ulll, bend right, "\delta_n \cdots \delta_0" sloped] \\
& \ord 0 \ular["\delta_0"] \urar["\delta_1"'] && \cdots \ular["\delta_0"] \urar["\delta_1"'] && \ord 0 \ular["\delta_0"] \urar["\delta_1"'] &
\end{tikzcd}
\]
Hence $X$ takes it to the following limit in $\Set$, so $X_{n + 2}$ has the expected form:
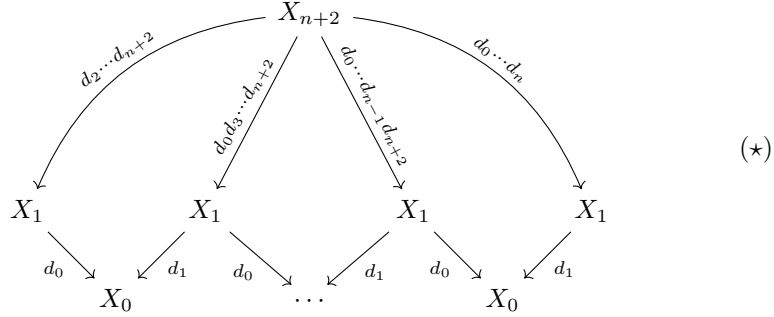
\begin{equation}
\label{eq:colimit}
\tag{$\star$}
\begin{tikzcd}[column sep=small]
&&& X_{n + 2} \ar[dlll, bend right, "d_2 \cdots d_{n + 2}" sloped] \dlar["d_0 d_3 \cdots d_{n + 2}" sloped] \drar["d_0 \cdots d_{n - 1} d_{n + 2}" sloped] \ar[drrr, bend left, "d_0 \cdots d_n" sloped] &&& \\[4em]
X_1 \drar["d_0"'] && X_1 \dlar["d_1"] \drar["d_0"'] && X_1 \dlar["d_1"] \drar["d_0"'] && X_1 \dlar["d_1"] \\
& X_0 && \cdots && X_0 &
\end{tikzcd}
\end{equation}
\end{proof}

\begin{lemma}
\label{lma:transitivity}
The face maps $d_i : X_n \to X_{n - 1}$ are given by
\[
d_i(x_0, \dots, x_n) = (x_0, \dots, x_{i - 1}, x_{i + 1}, \dots, x_n)
\]
In particular, $\leq_X$ must be transitive, as witnessed by $d_1 : X_2 \to X_1$.
\end{lemma}

\begin{proof}
This holds trivially for $n = 1$.
Now we have the following colimit in $\Delta$:
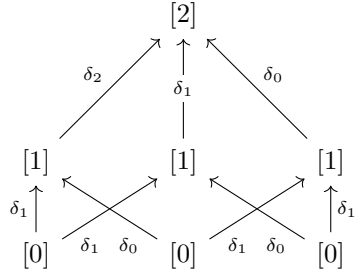
\[
\begin{tikzcd}[column sep=large]
& \ord 2 & \\[2em]
\ord 1 \urar["\delta_2"] & \ord 1 \uar["\delta_1" description] & \ord 1 \ular["\delta_0"'] \\
\ord 0 \uar["\delta_1"] \urar["\delta_1"' very near start] & \ord 0 \ular["\delta_0" very near start] \urar["\delta_1"' very near start] & \ord 0 \ular["\delta_0" very near start] \uar["\delta_1"']
\end{tikzcd}
\]
Hence $X$ takes it to the following limit diagram in $\Set$:
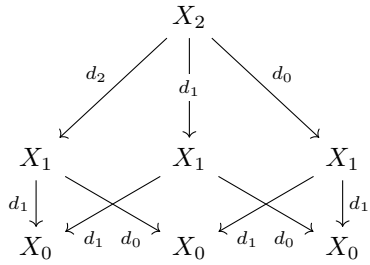
\[
\begin{tikzcd}[column sep=large]
& X_2 \dlar["d_2"'] \dar["d_1" description] \drar["d_0"] & \\[2em]
X_1 \dar["d_1"'] \drar["d_0"' very near end] & X_1 \dlar["d_1" very near end] \drar["d_0"' very near end] & X_1 \dlar["d_1" very near end] \dar["d_1"] \\
X_0 & X_0 & X_0
\end{tikzcd}
\]
Therefore the result holds for $n = 2$.
Now will show that the face maps for $n > 2$ arise as (co)limits of the face maps for $n = 1, 2$.
We considering three cases:
\begin{enumerate}
\item The following diagram is a pushout in $\Delta$:
\[
\begin{tikzcd}[column sep=4em]
\ord 0 \dar["\delta_0"'] \rar["\delta_{n + 2} \cdots \delta_1"] & \ord{n + 2} \dar["\delta_0"] \\
\ord 1 \rar["\delta_{n + 3} \cdots \delta_2"'] & \ord{n + 3} \arrow[ul, phantom, "\ulcorner" very near start]
\end{tikzcd}
\]
Hence the following diagram is a pullback in $\Set$:
\[
\begin{tikzcd}[column sep=4em]
X_{n + 3} \arrow[dr, phantom, "\lrcorner" very near start] \dar["d_0"'] \rar["d_2 \cdots d_{n + 3}"] & X_1 \dar["d_0"] \\
X_{n + 2} \rar["d_1 \cdots d_{n + 2}"'] & X_0
\end{tikzcd}
\]
The bottom map is one of the colimit legs in \eqref{eq:colimit}, so the left map must be
\[
(x_0, \dots, x_{n + 2}) \mapsto (x_1, \dots, x_{n + 2})
\]
\item The following diagram is a pushout in $\Delta$:
\[
\begin{tikzcd}[column sep=4em]
\ord 0 \dar["\delta_1"'] \rar["\delta_{n + 1} \cdots \delta_0"] & \ord{n + 2} \dar["\delta_{n + 3}"] \\
\ord 1 \rar["\delta_{n + 1} \cdots \delta_0"'] & \ord{n + 3} \arrow[ul, phantom, "\ulcorner" very near start]
\end{tikzcd}
\]
Hence the following diagram is a pullback in $\Set$:
\[
\begin{tikzcd}[column sep=4em]
X_{n + 3} \arrow[dr, phantom, "\lrcorner" very near start] \dar["d_{n + 3}"'] \rar["d_0 \cdots d_{n + 1}"] & X_1 \dar["d_1"] \\
X_{n + 2} \rar["d_0 \cdots d_{n + 1}"'] & X_0
\end{tikzcd}
\]
The bottom map is one of the colimit legs in \eqref{eq:colimit}, so the left map must be
\[
(x_0, \dots, x_{n + 2}) \mapsto (x_0, \dots, x_{n + 1})
\]
\item The following diagram is a pushout in $\Delta$ for $0 < i < n + 3$:
\[
\begin{tikzcd}[column sep=8em]
\ord 1 \dar["\delta_1"'] \rar["\delta_{n + 2} \cdots \delta_{i + 1} \delta_{i - 2} \cdots \delta_0"] & \ord{n + 2} \dar["\delta_i"] \\
\ord 2 \rar["\delta_{n + 3} \cdots \delta_{i + 2} \delta_{i - 2} \cdots \delta_0"'] & \ord{n + 3} \arrow[ul, phantom, "\ulcorner" very near start]
\end{tikzcd}
\]
Hence the following diagram is a pullback in $\Set$:
\[
\begin{tikzcd}[column sep=8em]
X_{n + 3} \arrow[dr, phantom, "\lrcorner" very near start] \dar["d_i"'] \rar["d_0 \cdots d_{i - 2} d_{i + 2} \cdots d_{n + 3}"] & X_2 \dar["d_1"] \\
X_{n + 2} \rar["d_0 \cdots d_{i - 2} d_{i + 1} \cdots d_{n + 2}"'] & X_1
\end{tikzcd}
\]
The bottom map is one of the colimit legs in \eqref{eq:colimit}, so the left map must be
\[
(x_0, \dots, x_{n + 2}) \mapsto (x_0, \dots, x_{i - 1}, x_{i + 1}, \dots, x_{n + 2})
\]
\end{enumerate}
\end{proof}

\begin{lemma}
\label{lma:reflexivity}
The degeneracy maps $s_i : X_n \to X_{n + 1}$ are given by
\[
s_i(x_0, \dots, x_n) = (x_0, \dots, x_i, x_i, \dots, x_n)
\]
In particular, $\leq_X$ must be reflexive, as witnessed by $s_0 : X_0 \to X_1$.
\end{lemma}

\begin{proof}
This holds for $n = 0$ because the following commutes:
\[
\begin{tikzcd}
X_0 \rar["s_0"] \drar["\operatorname{diag}"'] & X_1 \dar["\ang{d_1, d_0}"] \\
& X_0 \times X_0
\end{tikzcd}
\]
Now the following diagram is a pushout in $\Delta$:
\[
\begin{tikzcd}[column sep=8em]
\ord 1 \dar["\sigma_0"'] \rar["\delta_{n + 2} \cdots \delta_{i + 2} \delta_{i - 1} \cdots \delta_0"] & \ord{n + 2} \dar["\sigma_i"] \\
\ord 0 \rar["\delta_{n + 1} \cdots \delta_{i + 1} \delta_{i - 1} \cdots \delta_0"'] & \ord{n + 1} \arrow[ul, phantom, "\ulcorner" very near start]
\end{tikzcd}
\]
Hence the following diagram is a pullback in $\Set$:
\[
\begin{tikzcd}[column sep=8em]
X_{n + 1} \arrow[dr, phantom, "\lrcorner" very near start] \dar["s_i"'] \rar["d_0 \cdots d_{i - 1} d_{i + 1} \cdots d_{n + 1}"] & X_0 \dar["s_0"] \\
X_{n + 2} \rar["d_0 \cdots d_{i - 1} d_{i + 2} \cdots d_{n + 2}"'] & X_1
\end{tikzcd}
\]
The bottom map is one of the colimit legs in \eqref{eq:colimit}, so the left map must be
\[
(x_0, \dots, x_{n + 2}) \mapsto (x_1, \dots, x_i, x_i, \dots, x_{n + 2})
\]
\end{proof}

\begin{lemma}
\label{lma:antisymmetry}
The relation $\leq_X$ is antisymmetric.
\end{lemma}

\begin{proof}
The following diagram is a colimit in $\Delta$:
\[
\begin{tikzcd}[column sep=small]
& \ord 0 & \\
\ord 1 \urar["\sigma_0"] && \ord 1 \ular["\sigma_0"'] \\
\ord 0 \uar["\delta_1"] \ar[urr, "\delta_0"' very near start] && \ord 0 \ar[ull, "\delta_0" very near start] \uar["\delta_1"']
\end{tikzcd}
\]
Hence $X$ takes it to the following limit diagram in $\Set$:
\[
\begin{tikzcd}[column sep=small]
& X_0 \dlar["s_0"'] \drar["s_0"] & \\
X_1 \dar["d_1"'] \ar[drr, "d_0"' very near end] && X_1 \ar[dll, "d_0" very near end] \dar["d_1"] \\
X_0 && X_0
\end{tikzcd}
\]
In particular, this means that $X_0$ is isomorphic to the set of pairs $(x, y)$ such that $x \leq_X y$ and $y \leq_X x$, via the map $x \mapsto (x, x)$.
Hence $\leq_X$ is antisymmetric.
\end{proof}

\begin{proposition}
The relation $\leq_X$ is a partial order.
\end{proposition}

\begin{proof}
This follows from \Cref{lma:transitivity,lma:reflexivity,lma:antisymmetry}.
\end{proof}

\begin{theorem}
The nerve $N : \Pos \to [\Delta^{\op}, \Set]_{\cts}$ is an equivalence.
\end{theorem}

\begin{proof}
We already know that $N$ is fully faithful.
It is also essentially surjective: for every continuous simplicial set $X$, we have that $X \cong NP$ for $P = (X_0, \leq_X)$.
The isomorphism exists by \Cref{lma:chains} and is natural by \Cref{lma:transitivity,lma:reflexivity}.
\end{proof}

\begin{corollary}
$\Pos$ is the free conservative cocompletion of $\Delta$.
\end{corollary}

Therefore $\Pos$ satisfies the universal property of \Cref{def:cocompletion}.
That is, for every cocomplete category $\cat C$ and cocontinuous functor $F : \Delta \to \cat C$, there exists an essentially unique cocontinuous functor $\widetilde{F} : \Pos \to \cat C$ extending $F$:
\[
\begin{tikzcd}
\Delta \drar["F"'] \rar[hook] & \Pos \dar[dashed, "\widetilde{F}"] \\
& \cat C
\end{tikzcd}
\]
According to Kelly~\cite{kelly1982basic}, the functor $\widetilde{F}$ can be given by a Left Kan extension:
\[
\widetilde{F}(P) \coloneqq \colim(i \downarrow P \xrightarrow{\pi_\Delta} \Delta \xrightarrow{F} \cat C)
\]
Intuitively, this means that to compute $\widetilde{F}(P)$ for a poset $P$, we consider all finite chains of $P$, and take the colimit in $\cat C$ of the images of these chains under $F$.

\bibliographystyle{plainurl}
\bibliography{Posets}

\end{document}